\documentclass[a4paper,11pt,english]{article}

\usepackage{a4wide,bm,epsfig,booktabs}

\usepackage{amsmath}
\usepackage{tikz}
\usepackage{setspace}
\usepackage{relsize}  
\usepackage{graphicx}  
\usepackage{cancel} 
\usepackage{slashed}
\usepackage{verbatim} 
\usepackage{enumerate} 
\usepackage{stmaryrd} 
\usepackage{cite}
\usepackage[flushmargin,hang]{footmisc}
\usepackage{fancyhdr} 
\usepackage[arrow, matrix,curve]{xy}
\usepackage{hyperref}
\usepackage{amssymb}
\usepackage{amsthm}
\usepackage{eucal}
\usepackage{marvosym}
\usepackage{wasysym}

\DeclareMathOperator*{\innt}{\ThisStyle{\vstretch{0.9}{\hstretch{1.5}{\rotatebox{10}{$\SavedStyle\hspace{-0.5pt}\!\int\!\hspace{-0.5pt}$}}}}}
\DeclareMathOperator*{\Der}{\ThisStyle{\hstretch{1.2}{\rotatebox{0}{$\SavedStyle\delta^r$}}}}
\usepackage{scalerel}

\newcommand{\NN} {\mathbb{N}}

\newcommand{\RR} {\mathbb{R}}

\newcommand{\lip}  {\mathrm{lip}}

\newcommand{\uu} {\mathfrak{u}}
\newcommand{\mm} {\mathfrak{m}}

\newcommand{\oo} {\mathfrak{o}}
\newcommand{\pp} {\mathfrak{p}}
\newcommand{\qq} {\mathfrak{q}}

\newcommand{\mmm} {\mm}
\newcommand{\ppp} {\pp}
\newcommand{\qqq} {\qq}

\newcommand{\SEM} {\mathfrak{P}}
\newcommand{\mult}  {\mathrm{m}}

\newcommand{\DIDE} {\mathfrak{D}}

\newcommand{\EV} {\mathrm{Evol}}
\newcommand{\EVE} {{\mathrm{evol}}}

\newcommand{\chart} {\Xi}

\newcommand{\RT} {\mathrm{R}}

\newcommand{\Ad} {\mathrm{Ad}}

\newcommand{\DP} {\DIDE\mathrm{P}}

\newcommand{\conj}  {\mathbf{c}}

\newcommand{\dd} {\mathrm{d}}
\newcommand{\im} {\mathrm{im}}
\newcommand{\dom} {\mathrm{dom}}

\newcommand{\U}  {\mathcal{U}}
\newcommand{\V}  {\mathcal{V}}

\newcommand{\mg} {\mathfrak{g}}

\newcommand{\cp} {\circ}
\newcommand{\COMP} {\mathfrak{K}}

\newcommand{\compact} {\mathrm{C}}
\newcommand{\compacto} {\mathrm{K}}

\newcommand{\he} {\hspace{1pt}}

\renewcommand{\theenumi}{\arabic{enumi})} 
\renewcommand{\labelenumi}{\theenumi}

\let\origenumerate\enumerate
\def\enumerate{\origenumerate\itemsep0pt}
\let\origitemize\itemize
\def\itemize{\origitemize\itemsep0pt}

\newtheorem{theorem}{Theorem}
\newtheorem{proposition}{Proposition}
\newtheorem{lemma}{Lemma}
\newtheorem{corollary}{Corollary}
\newtheorem{remark}{Remark}

\makeatletter
\def\blfootnote{\gdef\@thefnmark{}\@footnotetext}
\makeatother

\begin{document}
\title{The Strong Trotter Property for\\ Locally $\mu$-convex Lie Groups}
\author{
  \textbf{Maximilian Hanusch}\thanks{\texttt{mhanusch@math.upb.de}}
  \\[1cm]
  Institut f\"ur Mathematik \\
  Lehrstuhl f\"ur Mathematik X \\
  Universit\"at W\"urzburg \\
  Campus Hubland Nord \\
  Emil-Fischer-Stra\ss e 31 \\
  97074 W\"urzburg \\
  Germany
}
\date{February 19, 2020}
\maketitle

\begin{abstract}  
We show that an infinite dimensional Lie group in Milnor's sense has the strong Trotter property if it is locally $\mu$-convex. This is a continuity condition imposed on the Lie group multiplication that generalizes the triangle inequality for locally convex vector spaces, and is equivalent to $C^0$-continuity of the evolution map on its domain. In particular, the result proven in this paper significantly extends the respective result obtained by Gl\"ockner in the context of measurable regularity.
\end{abstract}

\section{Introduction}
Let $G$ be an infinite dimensional Lie group in Milnor's sense, with exponential map $\exp\colon \mg\supseteq \dom[\exp]\rightarrow G$. We say that $G$ has the  
\emph{strong Trotter property} (cf.\ \cite{HGM}) if for each $\mu\in C^1([0,1],G)$ with $\mu(0)=e$ and $\dot\mu(0)\in \dom[\exp]$, we have 
\begin{align}
\label{sddssddszz}
	\textstyle\lim_n \mu(\tau/n)^n=\exp(\tau\cdot \dot\mu(0))\qquad\quad\forall\: \tau\in [0,\ell]
\end{align} 
uniformly for each $\ell>0$. As already shown in \cite{HGM}, this implies\footnote{Although in \cite{HGM} $\dom[\exp]=\mg$ is presumed, the proofs of the mentioned implications just carry over to the situation considered in this paper --   provided, of course, that the definitions given in \cite{HGM} for the (strong) commutator-, and the Trotter property are adapted in the obvious way.} the \emph{strong commutator property}, and also the  \emph{Trotter-}, and the \emph{commutator property} that are relevant, e.g., in representation theory of infinite dimensional Lie groups \cite{N1}. 
More importantly, Theorem I in \cite{HGM} states that $G$ has the strong Trotter property if it is $R$-regular. 
Now, $R$-regularity implies $C^0$-continuity of the evolution map, so that Theorem 1 in \cite{RGM} shows that $G$ is \emph{locally $\mu$-convex}. This condition has originally been introduced in \cite{HGGG}, and states that to each continuous seminorm $\uu$ on the modeling space $E$ of $G$, and to each chart $\chart\colon G\supseteq \U\rightarrow \V\subseteq E$ of $G$ around $e$ with $\chart(e)=0$, there exists a continuous seminorm $\uu\leq \oo$ on $E$, such that  $\chart^{-1}(X_1)\cdot {\dots}\cdot \chart^{-1}(X_n)\in \U$ and  
\begin{align}
\label{opdspoopdpod}
	(\uu\cp\chart)\big(\chart^{-1}(X_1)\cdot {\dots}\cdot \chart^{-1}(X_n)\big)\leq \oo(X_1)+{\dots}+\oo(X_n)
\end{align} 
holds for all $X_1,\dots,X_n\in E$ with $\oo(X_1)+{\dots}+\oo(X_n) \leq 1$. 
Evidently, this condition generalizes the triangle inequality for locally convex vector spaces; and, in general (without regularity presumptions on $G$) is equivalent to $C^0$-continuity of the evolution map on its domain, cf.\ Theorem 1 in \cite{RGM}. 
In this paper, we show that local $\mu$-convexity already suffices to ensure validity of \eqref{sddssddszz}, i.e., that we have the following theorem.
\begin{theorem}
\label{podspopods}
If $G$ is locally $\mu$-convex, then $G$ has  the strong Trotter property.
\end{theorem} 
In particular, this drops the presumptions made in \cite{HGM} on the domain of the evolution map, as well as 
the completeness presumptions made in \cite{HGM} on $\mg$.

\section{Preliminaries}
\label{dsdssd}
In this section, we fix the notations, and discuss the properties of the product integral (evolution map) that we shall need in Sect.\ \ref{dskdskjkjdskjdsds} to prove Theorem \ref{podspopods}. The proofs of the facts mentioned but not verified in this section can be found, e.g., in Sect.\ 3 and Sect.\ 4 in \cite{RGM}.

\subsection{Lie Groups}
In the following, $G$ will denote an infinite dimensional Lie group in Milnor's sense \cite{HG,HA,MIL,KHN} that is modeled over the Hausdorff locally convex vector space $E$, with corresponding system of continuous seminorms $\SEM$. 
We denote the Lie algebra of $G$ by $\mg$, fix a chart 
\begin{align*}
	\chart\colon G\supseteq \U\rightarrow \V\subseteq E
\end{align*}
with $\V$ convex, $e\in \U$, and $\chart(e)=0$; 
 and identify $\mg\cong E$ via $\dd_e\chart\colon \mg\rightarrow E$. Specifically, the latter condition means that we will write $\ppp(X)$ instead of $(\pp\cp\dd_e\chart)(X)$ for each $\pp\in \SEM$ and $X\in \mg$ in the following. We let $\mult\colon G\times G\rightarrow G$ denote the Lie group multiplication, $\RT_g:=\mult(\cdot, g)$ the right translation by $g\in G$, and $\Ad\colon G\times \mg\rightarrow \mg$ the adjoint action, i.e., we have
\begin{align*}
	\Ad_g(X):= \Ad(g,X) :=\dd_e\conj_g(X)\qquad\quad\text{with}\qquad\quad \conj_g\colon G\ni h\mapsto g\cdot  h\cdot g^{-1}\in G
\end{align*}
for each $g\in G$ and $X\in  \mg$.

\subsection{The Product Integral}
Let $\COMP:=\{[r,r']\subseteq \RR\: |\: r<r'\}$ denote the set of all proper compact intervals in $\RR$. 
The \emph{right logarithmic derivative} is given by
\begin{align*}
	\Der\colon C^1([r,r'],G)\rightarrow C^0([r,r'],\mg),\qquad \mu\mapsto \dd_\mu\RT_{\mu^{-1}}(\dot \mu)\qquad\qquad\forall\: [r,r']\in \COMP.
\end{align*}
We let 
$\DIDE:= \bigsqcup_{[r,r']\in \COMP}\DIDE_{[r,r']}$ with   
$\DIDE_{[r,r']}:= \Der(C^1([r,r'],G))$ for each $[r,r']\in \COMP$, and define 
\begin{align*}
	C_*^{1}([r,r'],G):=\{\mu\in C^{1}([r,r'],G)\:|\: \mu(r)=e\}
\end{align*}
as well as
\begin{align*}
	\EV\colon \DIDE_{[r,r']}\rightarrow C_*^{1}([r,r'],G),\qquad\Der(\mu)\mapsto \mu\cdot \mu(r)^{-1}\qquad\qquad\forall\: [r,r']\in \COMP.
\end{align*}  
The \emph{product integral} is given by
\begin{align*}
	\textstyle\innt_s^t\phi:= \EV\big(\phi|_{[s,t]}\big)(t)\in G\qquad\quad \forall \:[s,t]\subseteq \dom[\phi],\:\: \phi\in \DIDE,
\end{align*}
and we let $\innt\phi:=\innt_r^{r'}\!\phi$ for $\phi\in \DIDE$ with $\dom[\phi]=[r,r']$.
\begin{remark}
Evidently, for $[r,r']=[0,1]$, $\innt \phi$ just equals the ``small evolution map'', usually denoted by $\EVE$ in the literature. Moreover, $\innt \phi$ equals the Riemann integral for the case that $(G,\cdot)=(F,+)$ is the additive group of a Hausdorff locally convex vectors space $F$ -- The formulas \ref{kdsasaasassaas}--\ref{subst} below then just generalize the respective formulas for the Riemann integral.\hspace*{\fill}$\ddagger$
\end{remark}
\noindent
We have the following elementary identities:
\vspace{2pt}
\begingroup
\setlength{\leftmargini}{17pt}
{
\renewcommand{\theenumi}{\emph{\alph{enumi})}} 
\renewcommand{\labelenumi}{\theenumi}
\begin{enumerate}
\item
\label{kdsasaasassaas}
	\hspace{3pt}$\textstyle\innt_r^t \phi \cdot \innt_r^t\psi=\innt_r^t \phi+\Ad_{\innt_r^\bullet\phi}(\psi)\qquad\quad\quad\hspace{25.4pt}\forall\: \phi,\psi\in \DIDE_{[r,r']},\:\: t\in [r,r']$.
	\vspace{4pt}
\item
\label{kdskdsdkdslkds}
	\hspace{1pt}$\textstyle\big[\!\innt_r^t \phi\big]^{-1} \big[\innt_r^t\psi\big]=\innt_r^t\Ad_{[\innt_r^\bullet\phi]^{-1}}(\psi-\phi)\qquad\quad\forall\: \phi,\psi\in \DIDE_{[r,r']},\:\: t\in [r,r']$.
	\vspace{4pt}
\item
\label{pogfpogf}
\hspace{4pt}For $r=t_0<{\dots}<t_n=r'$ and $\phi\in \DIDE_{[r,r']}$, we have 
	\begin{align*}
		\textstyle\innt_r^t\phi=\innt_{t_{p}}^t\! \phi\cdot \innt_{t_{p-1}}^{t_{p}} \!\phi \cdot {\dots} \cdot \innt_{r}^{t_1}\!\phi\qquad\quad\forall\:t\in (t_p,t_{p+1}],\:\: p=0,\dots,n-1.
	\end{align*}
		\vspace{-15pt}
\item
\label{subst}
	\hspace{4pt}For $\varrho\colon [\ell,\ell']\rightarrow [r,r']$ 
of class $C^1$, we have
\begin{align*}
	 \textstyle\innt_r^{\varrho(\bullet)}\phi=\big[\innt_\ell^\bullet\dot\varrho\cdot (\phi\cp\varrho)\he\big]\cdot \big[\innt_r^{\varrho(\ell)}\phi\he\big]\qquad\quad\forall\:\phi\in \DIDE_{[r,r']}.
\end{align*} 
\end{enumerate}}
\endgroup
\noindent	
Next, for $X\in \mg$, we write $\phi_X$ for the constant map $[0,1]\ni t\mapsto X\in \mg$. If $\phi_X\in \DIDE_{[0,1]}$ holds, we define  
$\exp(X):= \innt \phi_X$. Evidently, we have $0\in \dom[\exp]$, and it is straightforward from \ref{subst} that, cf.\ Appendix \ref{App1} 
\begin{align}
\label{spodpodspodspods}
	X\in \dom[\exp]\qquad\quad\Longrightarrow\qquad\quad \RR_{\geq 0}\cdot X\subseteq \dom[\exp].
\end{align}
Finally, we let $\DP^0(\COMP,\mg):= \bigsqcup_{[r,r']\in \COMP}\DP^0([r,r'],\mg)$; where $\DP^0([r,r'],\mg)$ (for $[r,r']\in \COMP$) denotes the set of all maps $\phi\colon [r,r']\rightarrow \mg$ such that there exist $r=t_0<{\dots}<t_n=r'$, and $\phi[p]\in \DIDE_{[t_p,t_{p+1}]}$ with
\begin{align*}
	\phi|_{(t_p,t_{p+1})}=\phi[p]|_{(t_p,t_{p+1})}\qquad\quad\forall\: p=0,\dots,n-1.
\end{align*}  
In this situation, we define $\innt_r^r\phi:=e$, and let 
\begin{align}
\label{pofdofdopfd}
	\textstyle\innt_r^t\phi&\textstyle:=\innt_{t_{p}}^t \phi[p] \cdot \innt_{t_{p-1}}^{t_p} \phi[p-1]\cdot {\dots} \cdot \innt_{r}^{t_1}\phi[0]\qquad\quad \forall\: t\in (t_{p}, t_{p+1}],\:\: p=0,\dots,n-1.
\end{align} 
A standard refinement argument in combination with \ref{pogfpogf} then shows that this is well defined (cf.\ Sect.\ 4.3 in \cite{RGM}), i.e., independent of any choices we have made. 
It is furthermore not hard to see that for $\phi,\psi\in \DP^0([r,r'],\mg)$, we have 
 $\Ad_{[\innt_r^\bullet\phi]^{-1}}(\psi -\phi)\in \DP^0([r,r'],\mg)$ with 
\begin{align}
\label{dfdssfdsfd}
	\textstyle\big[\innt_r^t \phi\big]^{-1}\big[\innt_r^t \psi\big]=\innt_r^t \Ad_{[\innt_r^\bullet\phi]^{-1}}(\psi -\phi)
	\qquad\quad\forall\: t\in [r,r'].
\end{align} 
\subsection{Some Estimates}
We recall several facts (cf.\ Sect.\ 3.4.1 in \cite{RGM}).
\begingroup
\setlength{\leftmargini}{17pt}
{
\renewcommand{\theenumi}{\roman{enumi})} 
\renewcommand{\labelenumi}{\theenumi}
\begin{enumerate}
\item
\label{as1}
For each compact $\compact\subseteq G$, and each $\qq\in \SEM$, there exists some $\qq\leq \mm\in \SEM$, as well as $O\subseteq G$ open with $\compact\subseteq O$, such that
\begin{align*}
	\qqq\cp \Ad_g\leq \mmm\qquad\quad\forall\: g\in O.
\end{align*} 
\item
\label{as2}
Assume we are given $\mu\in C^1([0,1],G)$, as well as   
$\ell>0$ and $m\geq 1$ with $[0,\ell]\cdot [0,1/m]\subseteq [0,1]$. For $\tau\in [0,\ell]$, we define $\mu_{\tau}\colon [0,1/m]\ni t\mapsto \mu(\tau\cdot t)\in G$. The chain rule yields $\Der(\mu_\tau)(t)=\tau\cdot\Der(\mu)(\tau\cdot t)$ for each $t\in [0,1/m]$, so that the map
\begin{align}
\label{fdpopofdpofd}
	\alpha\colon [0,\ell]\times [0,1/m]\rightarrow \mg,\qquad (\tau,t)\mapsto \Der(\mu_\tau)(t)
\end{align}
is continuous.
\end{enumerate}}
\endgroup
\noindent
We say that $C^0([0,\ell],G)\supseteq\{\mu_n\}_{n\in \NN}\rightarrow \mu\in C^0([0,\ell],G)$ converges uniformly for $\ell>0$ if to each neighbourhood $U\subseteq G$ of $e$, there exists some $n_U\in\NN$ with
\begin{align*}
	\mu_n(t)\in\: U\cdot \mu(t)\: \cap\: \mu(t)\cdot U\qquad \quad \forall\: n\geq n_U,\:\: t\in [0,\ell].
\end{align*}
It is straightforward to see that
\begin{lemma}
\label{fdfddfd}
A sequence $C^0([0,\ell],G)\supseteq\{\mu_n\}_{n\in \NN}\rightarrow \mu\in C^0([0,\ell],G)$ converges uniformly if and only if to each neighbourhood $V\subseteq G$ of $e$, there exists some $n_V\in\NN$ with
\begin{align*}
	\mu_n(t)\in \mu(t)\cdot V\qquad \quad \forall\: n\geq n_V,\:\: t\in [0,\ell].
\end{align*} 
\end{lemma}
\begin{proof}
	The proof is elementary, and can be found in Appendix \ref{App2}.  
\end{proof}
\subsection{Continuity of the Integral}
\label{sssxyxyaaaayx}
As already mentioned in the introduction, Theorem 1 in \cite{RGM} shows that\footnote{Here, one additionally has to apply Lemma 15 in \cite{RGM}, together with the observation that this lemma also holds if the interval $[0,1]$ is replaced by some fixed arbitrary interval $[r,r']\in \COMP$ there.}   
local $\mu$-convexity \eqref{opdspoopdpod} is equivalent to continuity of the product integral on $\DIDE\cap C^k([r,r'],\mg)$ for any $k\in \NN\sqcup\{\lip,\infty\}$, and $[r,r']\in \COMP$ w.r.t.\ the $C^0$-topology, i.e., w.r.t.\ the seminorms 
\begin{align}
\label{kfdkfdkjbcbcbc}
	\textstyle\ppp_\infty(\phi):=\sup_{t\in [r,r']}\ppp(\phi(t))\qquad\quad\forall\: \phi\in \DIDE\cap C^k([r,r'],\mg)
\end{align}
for $\pp\in \SEM$. 
It was furthermore shown in \cite{RGM} that local $\mu$-convexity implies that the product integral is continuous at zero on $\DP^0(\COMP,\mg)$ w.r.t.\ the $L^1$-topology, i.e., that the following proposition holds, cf.\ Proposition 2 in \cite{RGM}.
\begin{proposition}
\label{aaapofdpofdpofdpofd}
Assume that $G$ is locally $\mu$-convex. Then, to each $\pp\in \SEM$, there exists some $\pp\leq \qq\in \SEM$, such that for each $\phi\in \DP^0(\COMP,\mg)$ we have
\begin{align*}
	\textstyle\int_r^{r'}\qqq(\phi(s))\:\dd s \leq 1
	\qquad\quad \Longrightarrow\qquad\quad	\textstyle(\pp\cp\chart)\big(\innt_r^\bullet\phi\big)\leq \int_r^\bullet \qqq(\phi(s))\:\dd s 
\end{align*} 
with $r,r'\in \RR$ such that $\dom[\phi]=[r,r']$.
\end{proposition}
Using \eqref{dfdssfdsfd}, this generalizes as follows.
\begin{lemma}
\label{podspodspods}
Assume that $G$ is locally $\mu$-convex, and let $\compacto\subseteq G$ be compact. Then, to each $\pp\in \SEM$,    
there exist $\pp\leq\mm\in \SEM$ and $O\subseteq G$ open with $\compacto\subseteq O$, such that for each $[r,r']\in \COMP$, we have
\begin{align*}
\textstyle(\pp\cp\chart)\big([\innt_r^\bullet\phi\big]^{-1}\big[\innt_r^\bullet \psi\big]\big)\leq \int_r^\bullet \mmm(\psi(s)-\phi(s))\:\dd s
\end{align*}
for all $\phi,\psi\in \DP^0([r,r'],\mg)$ with $\im[\innt_r^\bullet\phi] \subseteq O$ and $\int_r^{r'} \mmm(\psi(s)-\phi(s))\:\dd s\leq 1$. 
\end{lemma}
\begin{proof}
	For $\pp\in \SEM$ fixed, we choose $\pp\leq \qq$ as in Proposition \ref{aaapofdpofdpofdpofd}. 
Since $\compact:=\compacto^{-1}$ is compact, by \ref{as1}, there exists some $\qq\leq\mm\in \SEM$, as well as $O\subseteq G$ open with $\compacto\subseteq O$, such that
\begin{align}
\label{opgfpogfpo}
	\qqq\cp\Ad_{g^{-1}}\leq \mmm\qquad\quad\forall\:g\in O.
\end{align} 
Let now $\phi,\psi\in \DP^0([r,r'],\mg)$ be given, with $\im[\innt_r^\bullet\phi]\subseteq O$ and $\int_r^{r'}\mmm(\psi(s)-\phi(s))\:\dd s\leq 1$. We obtain from \eqref{opgfpogfpo} that
\begin{align}
\label{podspodspodsaaaa}
	\qqq(\chi)\leq \mmm(\psi-\phi)\qquad\text{holds for}\qquad \chi:= \Ad_{[\innt_r^\bullet \phi]^{-1}}(\psi-\phi),
\end{align}
hence, $\int_r^{r'}\qqq(\chi(s))\: \dd s\leq \int_r^{r'}\mmm(\psi(s)-\phi(s))\: \dd s\leq 1$. 
Then, Proposition \ref{aaapofdpofdpofdpofd} shows
\vspace{-6pt}
\begin{align*}
	\textstyle(\pp\cp\chart)\big(\innt_r^t\chi\big)\leq \int_r^t \qqq(\chi(s))\:\dd s\stackrel{\eqref{podspodspodsaaaa}}{\leq} \int_r^t \mmm(\psi(s)-\phi(s))\:\dd s\qquad\quad\forall\: t\in [r,r'].
\end{align*}
Since $\innt_r^t\chi$ equals the right hand side of \eqref{dfdssfdsfd}, we obtain
\begin{align*}
	\textstyle(\pp\cp\chart)\big([\innt_r^\bullet \phi\big]^{-1}\big[\innt_r^\bullet \psi\big]\big)\leq \int_r^\bullet \mmm(\psi(s)-\phi(s))\:\dd s,
\end{align*}
which proves the claim.
\end{proof}

\section{The Strong Trotter Property}
\label{dskdskjkjdskjdsds}
We now prove Theorem \ref{podspopods}. We start with the following observation.
\begin{lemma}
\label{poopo}
Let $G$ be locally $\mu$-convex, and assume that $L\cdot \phi\subseteq \DP^0([0,1],\mg)$ holds for $L\in \COMP$ and $\phi\colon [0,1]\rightarrow \mg$. Then, 
\begin{align}
\label{kjfdkjdfkjfdfd}
	\textstyle\Phi\colon L\times [0,1]\rightarrow G,\qquad (\tau,t)\mapsto \innt_0^t\tau\cdot \phi
\end{align}
is continuous, thus has compact image. 
\end{lemma}
\begin{proof}
Let $\tau\in L$, $t\in[0,1]$, and $h,h'\in [-1,1]$ be such that $\tau+[0,1]\cdot h\subseteq L$ and $t+[0,1]\cdot h'\subseteq [0,1]$ holds. Then, 
\begin{align*}
	\textstyle\mathrm{B}_{h}:=[\innt_0^t (\tau+h)\cdot \phi]\cdot [\innt_0^{t} \tau\cdot \phi]^{-1}=[\innt_0^{t} \tau\cdot \phi]\cdot\underbrace{\textstyle[\innt_0^{t} \tau\cdot \phi]^{-1}\cdot [\innt_0^t (\tau+h)\cdot \phi]}_{\mathrm{C}_h}\cdot\: [\innt_0^{t} \tau\cdot \phi]^{-1}
\end{align*}
tends to $e$ for $h\rightarrow 0$; because $\mathrm{C}_h$ tends to $e$ for $h\rightarrow 0$, by Lemma \ref{podspodspods} applied to $\compacto=\im[\innt_0^\bullet\tau\cdot \phi]$. We obtain from \ref{pogfpogf} that
\begin{align*}
		\textstyle\Phi(\tau+h,t+h')\cdot \Phi(\tau,t)^{-1}	&\textstyle= \big[\underbrace{\textstyle\innt_t^{t+h'} (\tau+h)\cdot \phi}_{\mathrm{A}^+_{h'}}\big]\hspace{16.1pt}\cdot\:\:\mathrm{B}_{h}
		\qquad\quad\text{holds for}\qquad\quad h'>0,\\
		\textstyle\Phi(\tau+h,t+h')\cdot \Phi(\tau,t)^{-1}&\textstyle= \textstyle\big[\underbrace{\textstyle\innt_{t-|h'|}^{t} (\tau+h)\cdot \phi}_{A_{h'}^-}\big]^{-1}\cdot\:\: \mathrm{B}_{h}
				\qquad\quad\text{holds for}\qquad\quad h'<0.\\[-20pt]
	\end{align*} 
	Since the integrands are bounded, Proposition \ref{aaapofdpofdpofdpofd} shows that $\lim_{h'\rightarrow 0} A^{\pm}_{h'}=e$ converges uniformly in $h$, from which the claim is clear.
\end{proof}
Combining Lemma \ref{poopo} with Lemma \ref{podspodspods}, we obtain the following corollary.
\begin{corollary}
\label{podspodspodsdd}
Let $G$ be locally $\mu$-convex, and assume that $L\cdot \phi\subseteq \DP^0([0,1],\mg)$ holds for $L\in \COMP$ and $\phi\colon [0,1]\rightarrow \mg$. Then, to each $\pp\in \SEM$,     
there exists some $\pp\leq\mm\in \SEM$, such that
\begin{align*}
 \textstyle(\pp\cp\chart)\big([\innt_0^\bullet \tau\cdot \phi\big]^{-1}\big[\innt_0^\bullet \psi\big]\big)\leq \int_0^\bullet \mmm(\psi(s)-\tau\cdot \phi(s))\: \dd s
\end{align*} 
holds for each $\tau\in L$ and $\psi\in \DP^0([0,1],\mg)$ with $\int_0^1 \mmm(\psi(s)-\tau\cdot \phi(s))\:\dd s\leq 1$.
\end{corollary}
\begin{proof}
Let $\Phi$ be defined by \eqref{kjfdkjdfkjfdfd}. Since Lemma \ref{poopo} shows that $\compacto:=\im[\Phi]$ is compact, the claim is clear from Lemma \ref{podspodspods}.
\end{proof}
We are ready for the proof of Theorem \ref{podspopods}.
\begin{proof}[Proof of Theorem \ref{podspopods}]
Let $\mu\in C^1([0,1],G)$ with $\mu(0)=e$ and $\dot\mu(0)\in \dom[\exp]$ be given. 
We fix $\ell>0$, let $X:=\dot\mu(0)$, and choose $m \geq 1$ so large that $\ell/m\leq 1$ and $\mu([0,\ell/m])\subseteq \dom[\chart]= \U$ holds. We obtain from \eqref{spodpodspodspods} that 
\begin{align}
\label{lkdflkdflkdffkld}
	\{X_\tau:=\tau\cdot X\:|\: \tau\in [0,\ell]\}\subseteq \dom[\exp]\qquad\:\:\text{holds, implying}\qquad\:\: [0,\ell]\cdot \phi_X\subseteq \DP^0([0,1],\mg),
\end{align}
so that the hypotheses of Corollary \ref{podspodspodsdd} are fulfilled for $L= [0,\ell]$ and $\phi= \phi_X$ there. 
	We proceed as follows.  
	\begingroup
\setlength{\leftmargini}{12pt}
{
\renewcommand{\theenumi}{{\rm \Alph{enumi})}} 
\renewcommand{\labelenumi}{\theenumi}
\begin{itemize}
\item
\label{as100}
For $\tau\in [0,\ell]$ and $n\geq m$, we define 
\begin{align*}
	\chi_{\tau,n}:=\Der(\mu_{\tau})|_{[0,1/n]}\qquad\quad\text{for}\qquad\quad\mu_{\tau}\colon [0,1/m]\ni t\mapsto \mu(\tau\cdot t),
\end{align*}
 and let $t_p:= p/n$ for $p=0,\dots,n$. We furthermore let 
	\begin{align*}
		\phi_{\tau,n}[p]\colon [t_p,t_{p+1}]\ni t\mapsto \chi_{\tau,n}(t-t_p)\qquad\quad\forall\: p=0,\dots, n-1,
\end{align*}			
	and define   
	 $\phi_{\tau,n}\in \DP^0([0,1],\mg)$ by putting $\chi_{\tau,n}$ n-times in a row, i.e., we let 
	 \begin{align*}
	\phi_{\tau,n}|_{[t_p,t_{p+1})}&:=\phi_{\tau,n}[p]|_{[t_p,t_{p+1})}\qquad\quad\forall\: p=0,\dots,n-2
\end{align*}
as well as $\phi_{\tau,n}|_{[t_{n-1},t_{n}]}:=\phi_{\tau,n}[n-1]$.
\item
\label{as10}
	For $\tau\in [0,\ell]$, $n\geq m$, and $0\leq p\leq n-1$, we apply \ref{subst} to $\varrho_p\colon [t_p,t_{p+1}]\ni t\mapsto t-t_p\in [0,1/n]$, and obtain
	\begin{align}
	\label{podspodsaaaaaaa}
		\textstyle\innt \phi_{\tau,n}[p]=\innt \chi_{\tau,n}\cp\varrho_p=\innt\dot\varrho_p\cdot \chi_{\tau,n}\cp\varrho_p\stackrel{\ref{subst}}{=}\innt\chi_{\tau,n}=\mu_\tau(1/n)=\mu(\tau/n).
	\end{align}		
	Then, \eqref{pofdofdopfd} provides us with 
	\begin{align}
	\label{as11}
	\textstyle \innt\phi_{\tau,n} \stackrel{\eqref{pofdofdopfd}}{=} \innt\phi_{\tau,n}[n-1]\cdot {\dots} \cdot\innt\phi_{\tau,n}[0]\stackrel{\eqref{podspodsaaaaaaa}}{=}\mu(\tau/n)^n.
	\end{align}	
\item
\label{as12}
For each $\tau\in [0,\ell]$, $n\geq m$, and $\mm\in \SEM$, we have
	\begin{align*}
		\textstyle\mmm_\infty(\phi_{\tau,n}-\phi_{X_\tau})\stackrel{\eqref{kfdkfdkjbcbcbc}}{=}\sup_{t\in[0,1/n]}\mmm(\chi_{\tau,n}(t)-\tau\cdot X)
	\end{align*}
	with $\chi_{\tau,n}(0)=\Der(\mu_{\tau})(0)=\dot\mu_{\tau}(0)=\tau\cdot \dot\mu(0)=\tau\cdot X$. 
	\vspace{3pt}	
	
	It follows\footnote{Recall that for each $\tau\in [0,\ell]$ and $n\geq m$, we have $\chi_{\tau,n}=\Der(\mu_{\tau})|_{[0,1/n]}$.} 
	from continuity of \eqref{fdpopofdpofd} in \ref{as2} that to each $\mm\in \SEM$ and $0<\varepsilon\leq 1$, there exists some $n_{\mm,\varepsilon}\geq m$ with
	\begin{align}
	\label{lkfdlkfd}
		\nonumber\mmm_\infty(\phi_{\tau,n}&-\phi_{X_\tau})< \varepsilon\qquad\qquad\hspace{26.2pt}\forall\:\tau\in [0,\ell],\:\: n\geq n_{\mm,\varepsilon}\qquad\\[2pt]
		\text{implying}\qquad\quad\textstyle\int_0^1 \mmm(\phi_{\tau,n}(s)&-\tau\cdot \phi_X(s))\:\dd s< \varepsilon\qquad\quad\forall\:\tau\in [0,\ell],\:\: n\geq n_{\mm,\varepsilon}.	
	\end{align}
\end{itemize}}
\endgroup
\noindent	 
	Let now $\pp\in \SEM$ and $0<\varepsilon\leq 1$ be fixed.  
	We choose $\pp\leq \mm\in \SEM$ as in Corollary \ref{podspodspodsdd} for $L= [0,\ell]$ and $\phi= \phi_X$ there (recall that $[0,\ell]\cdot \phi_X\subseteq \DP^0([0,1],\mg)$ holds by \eqref{lkdflkdflkdffkld}), and let $n_{\mm,\varepsilon}\geq m$ be as in \eqref{lkfdlkfd}. Since 
	\begin{align*}
		\textstyle\innt_0^t \tau\cdot \phi_X\stackrel{\ref{subst}}{=}\exp(t\cdot \tau\cdot X)\qquad\quad\forall\: t\in [0,1],\:\:\tau\in [0,\ell]
\end{align*}
holds, we obtain from \eqref{lkfdlkfd} and Corollary \ref{podspodspodsdd} that
	\vspace{-8pt}
	\begin{align*}
 \textstyle(\pp\cp\chart)\big(\exp(t\cdot \tau\cdot X)^{-1}\cdot \innt_0^t \phi_{\tau,n}\big)\leq \int_0^t \mmm(\phi_{\tau,n}(s)-\tau\cdot \phi_X(s))\: \dd s\stackrel{\eqref{lkfdlkfd}}{<} \varepsilon
\end{align*}
for each $t\in [0,1]$, $\tau\in[0,\ell]$, and $n\geq n_{\mm,\varepsilon}$. 
It is thus clear that to each open neighbourhood $V\subseteq G$ of $e$, there exists some $n_V\geq m$ with
\begin{align*}
	\textstyle\mu(\tau/n)^n\stackrel{\eqref{as11}}{=}\innt_0^1 \phi_{\tau,n}\in \exp(\tau\cdot \dot\mu(0))\cdot V\qquad\quad\: \forall\:n\geq n_V,\:\: \tau\in [0,\ell].
\end{align*}
The claim now follows from Lemma \ref{fdfddfd}.
\end{proof}
\section*{Acknowledgements}
The author thanks K.-H.\ Neeb for raising the question whether the strong Trotter property  holds in the locally $\mu$-convex context. He furthermore thanks K.-H.\ Neeb and H.\ Gl\"ockner for general remarks on a draft of the present article; as well as an anonymous JoLT referee for his helpful suggestions concerning the presentation of this article. 
This work has been supported by the Alexander von Humboldt Foundation of Germany.

\addtocontents{toc}{\protect\setcounter{tocdepth}{0}}
\appendix

\section{Appendix}

\subsection{Appendix}
\label{App1}
\begin{proof}[Proof of Implication \eqref{spodpodspodspods}]
It follows, e.g., from Lemma 11 in \cite{RGM} that for each $n\geq 1$, the constant map $\phi^n_X\colon [0,n]\ni t\mapsto X\in \mg$ is in $\DIDE_{[0,n]}$. Let $0<s\leq 1$ and $n\geq 1$ be fixed. Then, \ref{subst} applied to $\varrho\colon [0,1]\ni t\mapsto s\cdot n\cdot t \in [0,s\cdot n]$ gives
\begin{align*}
	\textstyle\innt_0^{s\cdot n} \phi^n_{X}\stackrel{\ref{subst}}{=}\innt_0^1 \phi_{s\cdot n\cdot X}=\exp(s\cdot n\cdot X)\qquad\quad\Longrightarrow\qquad\quad s\cdot n\cdot X\in \dom[\exp].
\end{align*} 
Since $\RR_{>0}=\bigcup_{n\geq 1}(0,n]$ holds, the claim follows.   
\end{proof}

\subsection{Appendix}
\label{App2}

\begin{proof}[Proof of Lemma \ref{fdfddfd}]
	The one direction is evident. For the other direction, 
	let $U\subseteq G$ be a fixed neighbourhood of $e$. To establish the proof, it suffices to show that there exists a neighbourhood $V\subseteq U$ of $e$ with
\begin{align}
\label{fdfdfdfdfd}
	\mu(t)\cdot V\subseteq U\cdot \mu(t)\qquad\quad\forall\: t\in [0,\ell].
\end{align}	
Since $\im[\mu]$ is compact, Theorem 4.9 in \cite{HAW} provides us with an identity neighbourhood $V\subseteq G$, such that $\mu(t)\cdot V\cdot \mu(t)^{-1} \subseteq U$ holds for each $t\in [0,\ell]$, which is equivalent to \eqref{fdfdfdfdfd}.    
\end{proof}

\end{document}